\documentclass[11pt]{amsart}
\usepackage{amssymb,amsmath,amsfonts,amscd,euscript}

\newcommand{\nc}{\newcommand}

\numberwithin{equation}{section}
\newtheorem{thm}{Theorem}[section]
\newtheorem{prop}[thm]{Proposition}
\newtheorem{lem}[thm]{Lemma}
\newtheorem{cor}[thm]{Corollary}
\theoremstyle{remark}
\newtheorem{rem}[thm]{Remark}

\newtheorem{example}[thm]{Example}

\newtheorem{conj}[thm]{{\bf Conjecture}}

\nc{\gl}{\mathfrak{gl}}
\nc{\GL}{\mathfrak{GL}}
\nc{\g}{\mathfrak{g}}
\nc{\gh}{\widehat\g}
\nc{\h}{\mathfrak{h}}
\nc{\la}{\lambda}
\nc{\al}{\alpha }
\nc{\be}{\beta }
\nc{\ve}{\varepsilon }
\nc{\om}{\omega }

\nc{\ta}{\theta}
\nc{\ch}{{\mathop {\rm ch}}}
\nc{\Tr}{{\mathop {\rm Tr}\,}}
\nc{\Id}{{\mathop {\rm Id}}}
\nc{\ad}{{\mathop {\rm ad}}}
\nc{\bra}{\langle}
\nc{\ket}{\rangle}
\nc{\x}{{\bf x}}
\nc{\bm}{{\bf m}}
\nc{\bs}{{\bf s}}
\nc{\ba}{{\bf a}}
\nc{\bb}{{\bf b}}
\nc{\bk}{{\bf k}}
\nc{\bp}{{\bf p}}
\nc{\pa}{\partial}
\nc{\ld}{\ldots}
\nc{\cd}{\cdots}
\nc{\hk}{\hookrightarrow}
\nc{\T}{\otimes}
\nc{\gr}{\mathrm{gr}}
\nc{\ov}{\overline}
\newcommand{\bin}[2]{\genfrac{(}{)}{0pt}{}{#1}{#2}}

\nc{\cO}{\mathcal O}
\nc{\msl}{\mathfrak{sl}}
\nc{\mgl}{\mathfrak{gl}}
\nc{\U}{\mathrm U}
\nc{\V}{\EuScript V}
\nc{\cL}{\mathcal{L}}
\nc{\Res}{\mathrm{Res\ }}

\newcommand{\bC}{{\mathbb C}}
\newcommand{\bQ}{{\mathbb Q}}
\newcommand{\bZ}{{\mathbb Z}}
\newcommand{\bN}{{\mathbb N}}

\newcommand{\fh}{{\mathfrak h}}
\newcommand{\fa}{{\mathfrak a}}
\newcommand{\fg}{{\mathfrak g}}
\newcommand{\fgh}{{\widehat{\mathfrak g}}}
\newcommand{\fb}{{\mathfrak b}}

\newcommand{\fn}{{\mathfrak n}}

\newcommand{\bfs}{{\bf{s}}}

\begin{document}

\title[]
{Nonsymmetric Macdonald polynomials, Demazure modules and PBW filtration}

\author{Evgeny Feigin}
\address{Evgeny Feigin:\newline
Department of Mathematics,\newline
National Research University Higher School of Economics,\newline
Vavilova str. 7, 117312, Moscow, Russia,\newline
{\it and }\newline
Tamm Theory Division, Lebedev Physics Institute
}
\email{evgfeig@gmail.com}
\author{Ievgen Makedonskyi}
\address{Ievgen Makedonskyi:\newline
Department of Mathematics,\newline
National Research University Higher School of Economics,\newline
Vavilova str. 7, 117312, Moscow, Russia
{\it and }\newline
Departments of Mechanics and Mathematics
Kiev Shevchenko National University
Vladimirskaya, 64, Kiev, Ukraine.
}
\email{}

\begin{abstract}
The Cherednik-Orr conjecture expresses the $t\to\infty$ limit of the nonsymmetric Macdonald polynomials
in terms of the PBW twisted characters of the affine level one Demazure modules. We prove this
conjecture in several special cases.
\end{abstract}

\maketitle

\section*{Introduction}
The  Macdonald symmetric functions $P_\la(x,q,t)$ \cite{M1} form a remarkable class of polynomials.
These polynomials depend on the variables $x=(x_1,\dots,x_n)$ and two parameters $q$ and $t$.
The  Macdonald symmetric functions can be specialized to the Hall-Littelwood polynomials ($q=0$),
to Schur polynomials ($q=t=0$) and to Jack symmetric polynomials ($q=t^\alpha$, $t\to 1$).

The polynomials $P_\la(x,q,t)$ have a nonsymmetric version $E_\la(x,q,t)$ (see \cite{Ch1}, \cite{O}, \cite{M2}).
The symmetric functions $P_\la(x,q,t)$ can be reconstructed from $E_\la(x,q,t)$
via certain symmetrization over the Weyl group. The nonsymmetric Macdonald polynomials
have many nice  and interesting properties. In particular, they are known to be
related to the representation theory of the affine Lie algebras (see \cite{S}, \cite{I}).
More precisely, the $t\to 0$ limit $E_\la(x,q,0)$ coincides with the character of the corresponding level one Demazure
module. In the recent papers \cite{CO1}, \cite{CO2}, \cite{OS}, \cite{CF} the $t\to\infty$ limit
of the nonsymmetric Macdonald polynomials was studied. In particular, it was shown
that $E_\la(x,q,\infty)$ are polynomials in $x$ and $q^{-1}$. Moreover, these polynomials have
non-negative coefficients. Thus it is natural to expect a relation with the representation theory.

Let $\fg$ be a simple Lie algebra and $\la$ be an anti-dominant weight.
We denote by $W_\la$ be the corresponding level one Demazure module with the extremal vector $w_\la$.
All Demazure modules are invariant with respect to the energy operator $d$
from the affine Kac-Moody algebra. An important special property of the Demazure
modules with antidominant highest weight is their invariance with respect to 
$\fg=\fg\T 1$. We assume that $dw_\la=0$ and thus the eigenvalues
of $d$ on $W_\la$ are nonnegative. We have the Kac-Moody character
$$\ch_{KM} W_\la=\sum_{k\ge 0} q^k\ch \{v\in W_\la, dv=kv\},$$
wher $\ch$ denotes the usual character with respect to the Cartan subalgebra of $\fg$.
One knows that $\ch_{KM} W_\la=E_\la(x,q,0)$ (see \cite{S}).

Cheredink and Orr conjectured that $E_\la(x,q,\infty)$ coincides with the PBW twisted
character of $W_\la$. To give the precise formulation of the conjecture, recall that the 
Demazure modules $W_\la$ are cyclic modules for the current algebra
$\fg\T\bC[t]$ with $w_\la$ being the cyclic vector. The PBW filtration on the universal enveloping
algebra of the current algebra induces the increasing filtration $F_s$ on the Demazure module. Each
space of this filtration is invariant with respect to the Cartan subalgebra and the associated
graded space $W_\la^{gr}$ is bi-graded by the operators
$d$ and by the PBW-grading operator $D$.
We have the PBW character
$$\ch_{PBW} W_\la=\sum_{k,s\ge 0} q^kp^s\ch \{v\in F_s/F_{s-1}, dv=kv\}.$$
Cherednik and Orr put forward the following conjecture \cite{CO1}, Conjecture 2.7:
\begin{conj}
Assume that $\la$ is an antidominant weight. Then
$$E_\la(x,q^{-1},\infty)=\ch_{PBW} W_\la|_{p=q}.$$
\end{conj}

Several checks on the level of representations of finite-dimensional algebra were worked out 
in \cite{CF}. The goal of this paper is to prove the conjecture in type $A$ in several cases. 
Namely, we prove the following theorem:
\begin{thm}\label{introthm}
Let $\fg$ be of type $A$. Then the Cherednik-Orr conjecture is true if the dual of $\la$ 
is equal to a multiple of a fundamental weight or 
to a linear combination of the first and the last fundamental weights.
\end{thm}

The paper is organized in the following way. In section 1 we collect main definitions and constructions
about Demazure modules and PBW filtration. We also derive PBW bases for special Demazure modules.
In section 2 we recall the Haglund-Haiman-Loehr formula \cite{HHL} for the nonsymmetric Macdonald polynomials,
derive the explicit combinatorial description of the  $t\to\infty$ limit and study the properties of the 
polynomials $E_\la(x,q^{-1},\infty)$.
Finally, in section 3, we prove Theorem \ref{introthm}.    

\section{Demazure modules and PBW filtration}
\subsection{Demazure modules and PBW filtration}
Let us briefly recall the main ingredients (see \cite{Kac}, \cite{Kum} for more details).

Let $\fg$ be a simple Lie algebra. We fix a Cartan decomposition $\fg=\fn^-\oplus\fb$, $\fb=\fn\oplus\fh$.
Let $\triangle_+$ be the set of positive roots of $\fg$, $n$ be the rank of $\fg$ and let $\al_i\in\triangle_+$, $i=1,\dots,n$
be the set of simple roots. We denote by $\om_i$, $i=1,\dots,n$ the fundamental weights.
Let $P=\bigoplus_{i=1}^n \bZ\om_i$ be the weight lattice and let $P_+=\bigoplus_{i=1}^n \bZ_{\ge 0}\om_i$
be the subset of dominant integral weights. Let $\la\in P_+$; we denote by $V_\la$ the irreducible
$\g$-module with highest weight $\la$.  
For $\al\in\triangle_+$, let $f_\al\in\fn^-$ and $e_\al\in\fn$ be the corresponding
Chevalley generators. 

For a Lie algebra $\fa$ we denote by $\fa[t]=\fa\T\bC[t]$ the corresponding current algebra.
We set $a[k]=a\T t^k\in\fa[t]$, $a\in\fa$, $k\ge 0$.

Let $\fgh=\fg\T\bC[t]\oplus\bC K\oplus\bC d$ be the affine Lie algebra; in particular,
$K$ is central and $[d,a\T t^k]=-ka\T t^k$. The current algebra $\fg[t]$ is naturally a
subalgebra of $\fgh$. We have the Cartan decomposition
\[
\fgh=\widehat{\fn^-}\oplus\widehat{\fh}\oplus\widehat{\fn}.
\]
For example, $\widehat{\fn}=\fg\T t\bC[t]\oplus\fn\T 1$, $\widehat{\fh}=\fh\T 1\oplus\bC K\oplus\bC d$.

Let $L=L_{\la,k}$ be an integrable irreducible highest weight $\fgh$ module with the highest weight
vector $v_{\la,k}$. The element  $K$ acts on $L$ as the scalar $k$ and this scalar is called the level of $L$.
The highest weight of $L_{\la,k}$ is the pair $(\la, k)$, where $\la\in P_+$ and $k$ is the
level of $L_{\la,k}$. We have the condition $(\la,\theta)\le k$, where $\theta$ is the highest root
of $\fg$. 

\begin{rem}
To make the pair $(\la,k)$ into an honest weight of $\gh$ one has to specify the eigenvalue of the energy 
operator $d$ on $v_{\la,k}$. However, this value is not important since the action of $d$ 
on a $\gh$ module can be shifted by an arbitrary scalar. We choose the convenient shift 
depending on a concrete situation.
\end{rem}

Let $W$ be the finite Weyl group of $\fg$ with the longest element $w_0$. For $\la\in P$ 
we denote the dual weight $w_0\la$ by $\la^*$. In particular, if $\la\in P_+$, then
$\la^*$ is the lowest weight of the irreducible $\fg$ module $V_\la$. Let $\widehat W$
be the corresponding affine Weyl group; 
thus, $\widehat W$ is the semi-direct product of $W$ with the root lattice. The finite
Weyl group naturally acts on the space of weights of $\fg$ and $\widehat W$
acts on the space of affine weights. For any integrable weigth $(\la,k)$ and $w\in \widehat W$
the weight space of the weight $w(\la,k)$ is one-dimensional. We fix one vector
in each corresponding space and denote it by $v_{w(\la,k)}$.
The Demazure module $D_w(\la)$ is defined as $D_w(\la)=\U(\widehat{\fn})v_{w(\la,k)}$.
We note that $D_w(\la)$ is not always invariant with respect to the action of $\fg=\fg\T 1$.

In what follows we only consider the level one modules. In this case for any weight $\mu\in P$
there exists unique integrable weight $(\la,1)$ and $w\in \widehat W$ such that $w(\la,1)=(\mu,1)$.
If $\mu$ is antidominant, i.e. $w_0\mu\in P_+$, then the Demazure module $D_w(\la)$ is $\fg\T 1$-invariant.
Assume that $w(\la,1)=(\mu,1)$. We use the shorthand notation $D_w(\mu)=W_\mu$,
$v_{w(\la,1)}=w_\mu$. 
In particular, one has ${\rm U}(\fn\T 1)w_\mu\simeq V_{\mu^*}$ with $w_\mu$ being the lowest
weight vector. One also has $W_\mu={\rm U}(\fn[t])w_\mu$. The modules $W_\mu$ play important role in representation theory and
in the theory of Macdonald polynomials
(see e.g. \cite{CL}, \cite{FL1}, \cite{FL2}, \cite{Kn}, \cite{S}, \cite{I}).
In particular, $W_\mu$ are Weyl modules and fusion modules for antidominant $\mu$.

Fixing $dw_\mu=0$, we obtain the Kac-Moody character of $W_\mu$, which is a polynomial in $q$:
\[
\ch_{KM}W_\mu=\sum_{r\ge 0} q^r\ch\{w\in W_\mu:\ dw=rw\},
\]
where $\ch$ is the $\fh$-character. In particular, $\ch_{KM} W_\mu|_{q=0}=\ch V_{\mu^*}$.

Let ${\rm U}(\fn[t])_s$ be the PBW filtration on the universal enveloping algebra.
Since $W_\mu={\rm U}(\fn[t])w_\mu$ we obtain the induced filtration on the Demazure module.
Let $W_\mu^{gr}$ be the associated graded module; thus
\[
W_\mu^{gr}=\bigoplus_{s\ge 0} W_\mu^{gr}(s),\ W_\mu^{gr}(s)=\frac{{\rm U}(\fn[t])_sw_\mu}{{\rm U}(\fn[t])_{s-1}w_\mu}.
\]
We note that $W_\mu^{gr}$ is a representation of the abelian Lie algebra $\fn^a[t]$, where $\fn^a$
is the abelian Lie algebra with the underlying vector space $\fn$.
Let $D$ be the PBW-degree operator on $W_\mu^{gr}$, i.e. $D|_{W^{gr}_\mu(s)}=s\cdot{\rm Id}$.
Let $W^{gr}_\mu(s,r)$
be the set of vectors $v\in W_\mu^{gr}(s)$ such that $dv=rv$. We note that each
$W_\mu(s,r)$ is naturally an $\fh$ module.
We define the PBW character of $W_\mu$ as
\[
\ch_{PBW}W^{gr}_\mu=\sum_{r,s\ge 0} q^rp^s\ch W^{gr}_\mu(s,r).
\]

\begin{rem}
The computation of the PBW character of $W^{gr}_\mu$ looks very interesting, but is out
of reach at the moment even in type $A$. One possible way to solve the problem 
is to find a basis of $W^{gr}_\mu$, i.e. a basis of $W_\mu$ compatible with the
PBW filtration (see \cite{FFL1},\cite{FFL2}, \cite{FFL3}, \cite{G} for the PBW bases of $V_\la$).
Below we describe the PBW bases for two special classes of Demazure modules.   
\end{rem}

\subsection{PBW basis}
Let $\fg=\msl_{n+1}$. Let $\al_{i,j}=\al_i+\dots+\al_j$ ($1\le i\le j\le n$)
be the set of positive roots. We denote by $f_{i,j}=f_{\al_{i,j}}$, $e_{i,j}=e_{\al_{i,j}}$  
the Chevalley generators of $\fg$.
Let $f_{i,j}[k]=f_{i,j}\T t^k$, $e_{i,j}[k]=e_{i,j}\T t^k$.
We list some properties of the Demazure modules $W_\la$ in the following lemma.
\begin{lem}
Let $\la^*=\sum_{i=1}^n m_i\om_i\in P_+$. Then 
\begin{itemize}
\item $W_\la$ is generated from the cyclic vector $w_\la\in W_\la$ by the action of the
operators $e_\al[k]=e_\al\T t^k$, $\al\in P_+$ and $k\ge 0$.
\item $\dim W_\la=\prod_{i=1}^n (\dim V_{\om_i})^{m_i}$, 
$W_\la\simeq \bigotimes_{i=1}^n V_{\om_i}^{\T m_i}$ as $\fg$-modules.
\item $W_\la$ is a $\fg\T\bC[t]$-module; it is isomorphic to the Weyl module.
\item $e_\al[k]w_\la=0$ for $k\ge\sum_{i} m_i$.
\end{itemize}
\end{lem}

In what follows it will be convenient to use the $\mgl_n$ notation for the characters of $V_\la$
and of $W_\la$: we represent the characters as functions in variables $x_1,\dots,x_{n+1}$.
For example, the character of the lowest weight vector $w_{\om_r^*}\in W_{\om_r^*}$ is equal
to $x_{n+1}\dots x_{n-r+2}$. In general, if $\la^*=\sum m_i\om_i$, then
\[
\ch w_\la=x_{n+1}^{m_1}(x_{n+1}x_n)^{m_2}\dots (x_{n+1}\dots x_2)^{m_n}=
\prod_{i=2}^{n+1} x_i^{\la_i},
\]
where $\la_i=m_n+m_{n-1}+\dots + m_{n-i+2}$.

\begin{prop}\label{rectprop}
Let $\la^*=m\om_r$. Then one has:
\begin{itemize}
\item The module $W_\la$ is generated from the vector $w_\la$ by the action of the polynomial 
algebra on variables $e_{i,j}[k]$, $i\le n-r+1\le j$, $k\ge 0$.
\item The PBW degree of a monomial $e_{i_1,j_1}[k_1]\dots e_{i_c,j_c}[k_c]$ is equal to $c$.
\item The PBW character and the Kac-Moody character are related by the formula
\begin{multline*}
\ch_{PBW} W_\la^{gr}(x_1,\dots,x_{n+1},p,q)=\\
\ch_{KM} W_\la(px_1,\dots,px_{n+1-r},x_{n+2-r},\dots,x_{n+1},q).
\end{multline*}
\end{itemize}
\end{prop}
\begin{proof}
The first claim is an immediate consequence of $e_\al w_\la=0$ if $(\la,\al)=0$. To prove the second statement
we note that the PBW degree of a vector $v\in W_\la$ is equal to the coefficient of $\al_{n+1-r}$ in the difference 
between the weight of $v$ and that of $w_\la$ (in each $e_{i,j}$, $i\le n-r\le j$ the simple root 
$\al_{n-r+1}$ shows up exactly once). The last claim follows from the observation that the character of $e_{i,j}$ 
is equal to $x_ix_{j+1}^{-1}$.  
\end{proof}

\begin{rem}
Proposition \ref{rectprop} implies that any basis for the Weyl module $W_{\om_r^*}$ is the PBW basis. 
\end{rem}

In the rest of the section we describe the PBW basis in the case $\la^*=m_1\om_1+m_2\om_n$.
We follow the notation from \cite{CL}.
Let $l\ge 0$ and let $\bfs=(\bfs(1)\le \dots \le \bfs(l))$ be a collection of nonnegative integers.
For a positive root $\al$ we use the notation
\[
e_\al(l,\bfs)=\prod_{1\le k\le l} e_\al[s(k)].
\]
If $\al=\al_{i,j}$, we abbreviate $e_{\al_{i,j}}(l,\bfs)$ by $e_{i,j}(l,\bfs)$.
We first recall several lemmas from \cite{CL}.

Let $\g=\msl_2$.
\begin{lem}\label{sl2}
The vectors $e_{1,1}(l,\bfs)w_{m\om_1}$ subject to the condition $\bfs(l)\le m-l$ form a basis of $W_{(m\om_1)^*}$.
The defining relations of  the $\msl_2[t]$-module $W_{m\om_1}$ are $f[k]w_{(m\om_1)^*}$ ($k\ge 0$), $h[k]w_{m\om_1}$ ($k> 0$),
$e[0]^{N}$ ($N>m$).
\end{lem}

\begin{lem}
Let $\la=(m\om_n)^*$. Then 
the vectors $\prod_{k=1}^n e_{1,k}(l_k,\bfs_k)w_{\la}$ subject to the conditions
\begin{gather*}
\bfs_1(l_1)\le m-l_1,\ \bfs_2(l_2)\le m-l_1-l_2,\dots, \bfs_{n-1}(l_n)\le m-l_1-\dots-l_n
\end{gather*}
form a PBW basis of $W_{\la}$.

Let $\la=(m\om_1)^*$. Then 
the vectors $\prod_{k=1}^n e_{k,n}(l_k,\bfs_k)w_{\la}$ subject to the conditions
\begin{gather*}
\bfs_n(l_n)\le m-l_n,\ \bfs_{n-1}(l_{n-1})\le m-l_n-l_{n-1},\dots, \bfs_1(l_1)\le m-l_n-\dots-l_1
\end{gather*}
form a PBW basis of $W_{\la}$.
\end{lem}
\begin{proof}
It is proved in \cite{CL} that the vectors described above form the PBW bases of $W_{(m\om_n)^*}$
and of $W_{(m\om_1)^*}$; these are PBW bases thanks to Proposition \eqref{rectprop}.
\end{proof}

We prove the following theorem.

\begin{thm}\label{main}
Let $\la^*=m_1\om_1+m_2\om_n$. Then the module $W_\la$ has a PBW basis of the form
\begin{equation}\label{basis}
e_{1,n}(l_{1,n},\bfs_{1,n})\prod_{k=1}^{n-1} e_{1,k}(l_{1,k},\bfs_{1,k})\prod_{k=2}^{n} e_{k,n}(l_{k,n},\bfs_{k,n})w_{\la}
\end{equation}
subject to the conditions
\begin{gather}
\label{1}
\bfs_{1,k}(l_{1,k})\le m_2-l_{1,1}-\dots-l_{1,k},\ k=1,\dots,n-1,\\
\label{n}
\bfs_{k,n}(l_{k,n})\le m_1-l_{n,n}-\dots-l_{k,n},\ k=n,\dots,2,\\
\label{1,n}
\bfs_{1,n}(l_{1,n})\le m_1+m_2-l_{1,1}-\dots -l_{1,n-1}-l_{1,n}-l_{2.n}-\dots-l_{n,n}.
\end{gather}
\end{thm}

\begin{lem}\label{dim}
The number of solutions of inequalities \eqref{1}, \eqref{n}, \eqref{1,n}  is equal to the dimension of $W_\la$.
\end{lem}
\begin{proof}
We know that $\dim W_{\la}=(n+1)^{m_1+m_2}$. So we only need to show that
\[
\sum_{\substack{l_{1,1}+\dots+l_{1,n-1}\le m_2\\ l_{n,n}+\dots+l_{n-1,n}\le m_1}} 
2^{m_1+m_2-l_{1,1}+\dots+l_{1,n-1}-l_{n,n}+\dots+l_{n-1,n}}
\bin{m_2}{l_{1,1},\dots,l_{1,n-1}} \bin{m_1}{l_{n,n},\dots,l_{n-1,n}}
\]
is equal to $(n+1)^{m_1+m_2}$, which is clear.
\end{proof}

Let $F_s$ be the PBW filtration on $W_\la$. For any $\al\in\triangle_+$ and $k\ge 0$ one has
$f_\al[k]F_s\subset F_s$. Hence we obtain the induced PBW-degree zero operators on $W^{gr}_\la$, 
which we denote by $\pa_\al[k]$. 
Recall that $W^{gr}_\la$ can be represented as the quotient of the polynomial ring in variables
$e_\al[k]$. We have the following easy lemma:
\begin{lem}
The operators $\pa_\al[k]$ are induced by the differential operators (which we also denote by 
$\pa_\al[k]$) on the
polynomial algebra in variables $e_\al[k]$. One has
$\pa_\al[k] e_\beta[r]=0$ unless $[f_\al,e_\beta]=c_{\al,\beta}^\gamma e_\gamma$
for some $\gamma\in\triangle_+$. If this equality holds, then
$\pa_\al[k] e_\beta[r]=c_{\al,\beta}^\gamma e_\gamma[k+r]$.
\end{lem}

Now let $\theta$ be the highest root of $\fg$. Let $\msl_2^\theta$ be the $\msl_2$ algebra
generated by $f_\theta$ and $e_\theta$.

\begin{lem}
The differential operators $\pa_\theta[k]$ vanish. The operators $f_\theta[k]$ map
$F_s$ to $F_{s-1}$ and hence induce the degree minus one operators $f^{gr}_\theta[k]$ on $W_\la^{gr}$. 
The operators $f^{gr}_\theta[k]$, $e_\theta[k]$ and $h_\theta[k]$ form the Lie algebra $\msl_2^\theta[t]$,
isomorphic to $\msl_2[t]$.
\end{lem}
\begin{proof}
We note that $[f_\theta,e_\al]\in\fb$ for any $\al\in\triangle_+$ and hence $f_\theta[k] F_s\subset F_{s-1}$.
Therefore we obtain that the operators $\pa_\theta[k]$ vanish, but there exists degree minus one operators 
$f_\theta^{gr}[k]$ on $W_\la^{gr}$. The last statement is clear.
\end{proof}

We now prove the main theorem. We first sketch the proof in the $\msl_3$-case, and then give the proof for general $n$.
\begin{lem}\label{sl3}
Let $\fg=\msl_3$, $\la^*=m_1\om_1+m_2\om_2$. Then the vectors
\begin{equation}\label{monom}
e_{1,1}(l_{1,1},\bfs_{1,1})e_{2,2}(l_{2,2},\bfs_{2,2})e_{1,2}(l_{1,2},\bfs_{1,2})w_\la
\end{equation}
subject to the conditions
\begin{equation}\label{cond}
s_{1,1}(l_{1,1})\le m_2-l_{1,1},\  s_{2,2}(l_{2,2})\le m_1-l_{2,2},\
s_{1,2}(l_{1,2})\le m_1+m_2-l_{1,1}-l_{2,2}-l_{1,2}
\end{equation}
form a basis of $W_\la^{gr}$.
\end{lem}
\begin{proof}
Consider an arbitrary vector of the form \eqref{monom}. We want to show that it can be rewritten as
a linear combination of monomials subject to conditions \eqref{cond}. Restricting
the module $W_\la$ to the subalgebras $\msl_2[t]$ corresponding to simple roots
we can assume that
\[
s_{1,1}(l_{1,1})\le m_2-l_{1,1},\  s_{2,2}(l_{2,2})\le m_1-l_{2,2}
\]
(since we know that these restrictions produce basis in the $\msl_2$ case, see Lemma \ref{sl2}).
Now we note that
\begin{equation}\label{rel}
e_{1,2}[0]^m e_{1,1}(l_{1,1},\bfs_{1,1})e_{2,2}(l_{2,2},\bfs_{2,2})w_\la=0
\end{equation}
provided $m+l_{1,1}+l_{2,2}>m_1+m_2$ (this can be shown via applying the differential operators, see the
proof in the general case below). Now consider the action of the algebra $\msl_2^\theta[t]$.
We note that 
\[
f_{1,2}[k] e_{1,1}(l_{1,1},\bfs_{1,1})e_{2,2}(l_{2,2},\bfs_{2,2})w_\la=0.
\]
Indeed, $[f_{1,2},e_{1,1}]=f_{2,2}$ and $[f_{1,2},e_{2,2}]=-f_{1,1}$. Therefore,
the PBW degree of $f_{1,2}[k] e_{1,1}(l_{1,1},\bfs_{1,1})e_{2,2}(l_{2,2},\bfs_{2,2})w_\la$
is at most $l_{1,1}+l_{2,2}-2$.
Now
using \eqref{rel} and the action $\msl_2[t]$ we obtain the desired claim thanks to Lemma \ref{sl2}.
\end{proof}

Now we give the general proof.
\begin{proof}
Because of Lemma \ref{dim} it suffices to prove that any vector in $W_\la$ can be written as a linear
combination of vectors \eqref{basis} subject to the conditions \eqref{1}, \eqref{n}, \eqref{1,n}.
First, let us restrict $W_\la$ to the subalgebra $\msl_{n-1}[t]$, corresponding to simple roots
$\al_1,\dots,\al_{n-1}$. Then we have all the relations from $W_{(m_2\om_n)^*}$ and hence we can assume
that all the conditions \eqref{1} hold. Similarly we can assume that restrictions \eqref{n} are satisfied.

We have
\begin{equation}\label{zero}
e_{1,n}[0]^{m}\prod_{j=1}^{n-1} e(l_{1,j},\bfs_{1,j})\prod_{i=2}^{n} e(l_{i,n},\bfs_{i,n})w_\la=0
\end{equation}
provided $m+\sum_{j=1}^{n-1} l_{1,j}+\sum_{i=2}^n l_{i,n}>m_1+m_2$.
In fact, we know that $e_{1,n}[0]^m w_\la=0$ if $m>m_1+m_2$; in addition
\begin{gather*}
\pa_{1,i-1}[k]^r e_{1,n}[0]^m = {\rm const}.\ e_{i,n}[k]^r e_{1,n}[0]^{m-r},\\
\pa_{j+1,n}[k]^r e_{1,n}[0]^m = {\rm const}.\ e_{1,j}[k]^r e_{1,n}[0]^{m-r}
\end{gather*}
which proves \eqref{zero}. The difference with the $\msl_3$ case is that $f^{gr}_\theta[k]$
does not always kill $\prod_{j=1}^{n-1} e(l_{1,j},\bfs_{1,j})\prod_{i=2}^{n} e(l_{i,n},\bfs_{i,n})w_\la$.
To resolve this difficulty we introduce a filtration $G_\bullet$ on $W_\la^{gr}$. The spaces $G_\nu$ 
of the filtration are
labeled by the weights $\nu\in P$; $G_\nu$ is spanned by the vectors of the form
\eqref{basis} subject to conditions \eqref{1} , \eqref{n} (but not \eqref{1,n}) such that the weight of 
\[
\prod_{k=1}^{n-1} e_{1,k}(l_{1,k},\bfs_{1,k})\prod_{k=2}^{n} e_{k,n}(l_{k,n},\bfs_{k,n})w_{\la}
\]
is at most $\la-\nu$. Then in the associated graded module we have the action of the algebra 
$\msl_2^\theta[t]$ and the argument of Lemma \ref{sl3} applies.
\end{proof}

\section{Nonsymmetric Macdonald polynomials}
\subsection{The Haglund-Haiman-Loehr formula}
Let $\lambda=(\lambda_1, \dots,\lambda_n)$ be an $n$-tuple of integers. The type $A$ nonsymmetric Macdonald polynomials
$E_\la(x,q,t)$ are polynomials in variables $x=(x_1,\dots,x_n)$ with coefficients in $\bQ(q,t)$.
They are simultaneous eigenfunctions of the Cherednik operators (see e.g. \cite{HHL}). In what follows
we need the following Knop-Sahi property of the Macdonald polynomials. Let
\begin{gather}
\label{pi}
\pi(\la_1,\dots,\la_n)=(\la_n+1,\la_1,\dots,\la_{n-1}),\\
(\Psi f)(x_1,\dots,x_n)=x_1f(x_2,\dots,x_n,q^{-1}x_1).
\end{gather}
Then $E_{\pi(\la)}(x,q,t)=q^{\la_n}\Psi E_\la(x,q,t)$.

We use the explicit combinatorial description of the nonsymmetric Macdonald polynomials from \cite{HHL}.
For a composition $\lambda=(\lambda_1, \dots,\lambda_n)$ the column diagram $dg'(\la)$ is the set
\[
dg'(\la)=\{(i,j)\in \bN^2:\ 1\le i\le n, 1\le j\le\la_i\}
\]
The augmented diagram $\widehat{dg}(\lambda)$ is defined by
\[
\widehat{dg}(\lambda)=dg'(\la)\cup\{(i,0): 1\le i\le n\},
\]
i.e. one box is added to the bottom of each column.
For example, for the composition $\lambda=(3,1,0,2,0,4)$ one has the following diagrams for 
$dg'(\lambda)$ and $\widehat{dg}(\lambda)$:
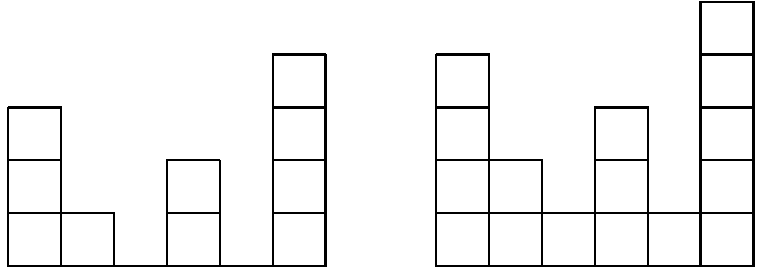
\[
\begin{picture}(140.00,130.00)(-5,00)

 \put(0.00,10.00){\line(1,0){120}}
 \put(0.00,30.00){\line(1,0){40}}
 \put(60.00,30.00){\line(1,0){20}}
 \put(100.00,30.00){\line(1,0){20}}

 \put(0.00,50.00){\line(1,0){20}}
 \put(0.00,70.00){\line(1,0){20}}
 \put(60.00,50.00){\line(1,0){20}}
 \put(100.00,50.00){\line(1,0){20}}
 \put(100.00,70.00){\line(1,0){20}}
 \put(100.00,90.00){\line(1,0){20}}

 \put(0.00,10.00){\line(0,1){60}}
 \put(20.00,10.00){\line(0,1){60}}
 \put(40.00,10.00){\line(0,1){20}}
 \put(60.00,10.00){\line(0,1){40}}
 \put(80.00,10.00){\line(0,1){40}}
 \put(100.00,10.00){\line(0,1){80}}
 \put(120.00,10.00){\line(0,1){80}}
\end{picture}\qquad 
\begin{picture}(140.00,130.00)(-5,00)

 \put(0.00,10.00){\line(1,0){120}}
 \put(0.00,30.00){\line(1,0){120}}
 \put(0.00,50.00){\line(1,0){40}}
 \put(60.00,50.00){\line(1,0){20}}
 \put(100.00,50.00){\line(1,0){20}}

 \put(0.00,70.00){\line(1,0){20}}
 \put(0.00,90.00){\line(1,0){20}}
 \put(60.00,70.00){\line(1,0){20}}
 \put(100.00,70.00){\line(1,0){20}}
 \put(100.00,90.00){\line(1,0){20}}
 \put(100.00,110.00){\line(1,0){20}}

 \put(0.00,10.00){\line(0,1){80}}
 \put(20.00,10.00){\line(0,1){80}}
 \put(40.00,10.00){\line(0,1){40}}
 \put(60.00,10.00){\line(0,1){60}}
 \put(80.00,10.00){\line(0,1){60}}
 \put(100.00,10.00){\line(0,1){100}}
 \put(120.00,10.00){\line(0,1){100}}
\end{picture}
\]
In what follows we will be mostly interested in anti-dominant diagrams, i.e. such that
$\la_i \leq \la_j$, if $i < j$.

A filling of $\la$ is the map $\sigma: dg'(\la)\to\{1,\dots,n\}$.
The associated augmented filling  $\widehat{\sigma}: \widehat{dg}(\lambda)\to\{1,\dots,n\}$
agrees with $\sigma$ on $dg'(\la)$ and $\widehat{\sigma}((j,0))=~j$ for $j=1,\dots,n$.

Two boxes are called attacking if they are in the same row or
they are in consecutive rows, and the box in the lower row is to the right of the one
in the upper row, i.e., they have the form $(i, j),(i_1, j-1)$ with $i < i_1$.
A filling is called non-attacking, if there are no equal elements in attacking boxes.

For a box $u$ let $d(u)$ be the box strictly below $u$. The descent of a filling
$\sigma$ is the set of  boxes $u$ such that $\widehat\sigma(u)>\widehat\sigma(d(u))$.

For a box $u=(i,j)$ let $leg(u)$ be the set of cells above $u$ and $l(u)=|leg(u)|=\la_i-j$.
We will also need the value $a(u)$ counting the cardinality of the arms of $u$.
Define:
\[arm^{left}(u)=\lbrace (i',j)\in dg'(\la)| i'<i, \lambda_{i'} \leq \lambda_i \rbrace.\]
\[arm^{right}(u)=\lbrace (i',j-1)\in \widehat{dg}(\la)| i'>i, \lambda_{i'} < \lambda_i \rbrace.\]
\[arm(u)=arm^{left}(u) \cup arm^{right}(u)\]
Then $a(u)=|\{arm(u)\}|$.

Note that for anti-dominant diagrams we have:
\[arm(u)=arm^{left}(u)=\lbrace (i',j)\in dg'(\la)| i'<i \rbrace.\]

Let $Des(\widehat\sigma)$ be the set of descents of $\widehat\sigma$ and
\[maj(\widehat\sigma)=\sum_{u \in Des(\widehat\sigma)}(l(u)+1).\]

A pair of attacking elements $u=(i,j)$ and $u'=(i', j')$ of $\widehat\sigma$ is an inversion if
$\widehat\sigma(u) < \widehat\sigma(u')$ and $j=j', i<i'$ or $j+1=j', i>i'$.
Let $Inv(\widehat\sigma)$ be the set of inversions of $\widehat\sigma$.
Let
\begin{equation}
coinv(\widehat\sigma)=\sum_{u \in dg'(\la)}a(u)-|Inv(\widehat\sigma)|+ 
|\lbrace (i<j:\lambda_i \leq \lambda_j) \rbrace|+\sum_{u \in Des(\widehat\sigma)}a(u).
\end{equation}
For a filling $\sigma$ of the diagram $dg'(\la)$ we define 
\[
x^\sigma=\prod_{i=1}^n x_i^{|\{u\in dg'(\la):\ \sigma(u)=i\}|}.
\]

\begin{thm}(Haglund, Haiman, and Loehr)
\begin{equation}\label{Macdonald}
E_\la(x; q,t)=\sum_{\sigma ~non-attacking}x^{\sigma}q^{maj (\widehat\sigma)}t^{coinv(\widehat\sigma)}
\prod_{\substack{u \in dg'(\la)\\ \widehat\sigma(u) \neq \widehat\sigma(d(u))}} \frac{1-t}{1-q^{l(u)+1}t^{a(u)+1}}.
\end{equation}
\end{thm}

In our paper we only need the special compositions $\mu$ of the form 
$\mu=\pi^{r}(\lambda)$, where $\lambda=(\la_1\le\dots\le\la_n)$ is anti-dominant
and $\pi^r$ is the power of the operator \eqref{pi}.

For an anti-dominant $\lambda$ we have:
\begin{equation}
coinv(\widehat\sigma)=\sum_{u \in dg'(\la)}a(u)-|Inv(\widehat\sigma)|+ \frac{n(n-1)}{2}+\sum_{u \in Des(\widehat\sigma)}a(u).
\end{equation}

Note that we have $\frac{n(n-1)}{2}$ inversions in the bottom row of $\widehat{dg}(\la)$.
Let $Inv'(\widehat\sigma)=Inv(\widehat\sigma)-\frac{n(n-1)}{2}$.
We have
\begin{equation}
coinv(\widehat\sigma)=\sum_{u \in dg'(\la)}a(u)-|Inv'(\widehat\sigma)|+\sum_{u \in Des(\widehat\sigma)}a(u).
\end{equation}

\begin{rem}
We note that for non-attacking $\sigma$ there are no inversions in $\widehat\sigma$ between the cells in the first and zero rows.
\end{rem}

For a composition $\mu=\pi^{r}(\lambda)$ with antidominant $\la$ one has  
\[
|\lbrace (i<j:\mu_i \leq \mu_j) \rbrace|=\frac{r(r-1)}{2}+\frac{(n-r)(n-r-1)}{2}
\]
and hence
\begin{equation}\label{coinv}
coinv(\widehat\sigma)=\sum_{u \in dg'(\la)}a(u)-|Inv'(\widehat\sigma)|-r(n-r)+\sum_{u \in Des(\widehat\sigma)}a(u).
\end{equation}

\subsection{ $t\to\infty$ limit}
In this section we give a combinatorial formula for the $t\to\infty$ limit of the polynomials
$E_\mu(x; q,t)$.

We note that $\lim_{t \to \infty}E_\mu(x; q,t)=
\lim_{t \to \infty}\sum_{\sigma ~non-attacking} x^\sigma t^{T(\sigma)}q^{Q(\sigma)}$, where
\begin{gather}
T(\sigma)=coinv(\widehat\sigma) - \sum_{\substack{u \in dg'\\ \widehat\sigma(u) \neq \widehat\sigma(d(u))}} a(u),\\
Q(\sigma)= maj (\widehat\sigma) - \sum_{\substack{u \in dg'\\ \widehat\sigma(u) \neq \widehat\sigma(d(u))}} (l(u)+1)=
-\sum_{\substack{u \in dg'\\ \widehat\sigma(u) < \widehat\sigma(d(u))}} (l(u)+1).
\end{gather}


For any cell $u=(i,j) \in \widehat{dg}(\lambda)$ we define the cell $\pi(u)\in dg'\left(\pi(\lambda)\right)$ by
\[
\pi(u)=\begin{cases} (i+1,j),\text{ if } i\neq n,\\ (1,j+1), \text{ if } i=n.\end{cases}
\]
We note that $v \in arm(u)$ iff $\pi(v) \in arm(\pi(u))$ and  $v \in leg(u)$ iff $\pi(v) \in leg(\pi(u))$;
in particular, $a(u)=a(\pi(u))$ and  $l(u)=l(\pi(u))$.

The following lemma (even for general $\mu$) is known (see \cite{CO1}). However we give a combinatorial
proof below. The construction given in the proof will be very important later.

\begin{lem}\label{inequalitypoweroft}
Let $\mu=\pi^r(\la)$ and assume that $\la$ is antidominant. Then  $\lim_{t \rightarrow \infty}E_\mu(x; q,t)$ is well defined
polynomial in $x$ and $q^{-1}$.
\end{lem}
\begin{proof}
We have to prove that for any non-attacking filling $\sigma$ the power
\begin{equation}\label{degt}
coinv(\widehat\sigma)-\sum_{\substack{u \in dg'(\mu)\\ \widehat\sigma(u) \neq \widehat\sigma(d(u))}}{a(u)} 
\end{equation}
is less than or equal to zero.
Using \ref{coinv} we have:
\begin{multline*}
coinv(\widehat\sigma)-\sum_{\substack{u \in dg'(\mu)\\ \widehat\sigma(u) \neq \widehat\sigma(d(u))}}a(u)=\\
\sum_{u \in dg'(\mu)} a(u)-|Inv'(\widehat\sigma)|-r(n-r)+\sum_{u \in Des(\widehat\sigma)}a(u)-
\sum_{\substack{u \in dg'(\mu)\\ \widehat\sigma(u) \neq \widehat\sigma(d(u))}} {a(u)}=\\
\sum_{u \in dg'(\mu)}a(u)-|Inv'(\widehat\sigma)|-r(n-r)-
\sum_{\substack{u \in dg'(\mu)\\ \widehat\sigma(u) < \widehat\sigma(d(u))}}{a(u)}.
\end{multline*}

We consider three cells $u=\pi^r(\tilde{u})$, $d(u)$ and $v \in arm(u)$ (i.e. $v=\pi^r(\tilde{v})$
for some $\tilde v \in arm(\tilde u)$).
Assume that $\widehat\sigma(u)>\widehat\sigma(d(u))$. Then we have $\widehat\sigma(v)<\widehat\sigma(u)$ 
or $\widehat\sigma(v)>\widehat\sigma(d(u))$.
Hence if $\widehat\sigma(u)>\widehat\sigma(d(u))$ then for any element in $arm(u)$ we have at least one inversion.
Let $Inv^\pi$ be the set of inversions of the form $\pi^r(u), \pi^r(v)$.
We have:
\begin{multline*} \sum_{u=\pi^r(\tilde{u})}a(u)-|Inv^\pi(\sigma)|-
\sum_{\substack{u=\pi^r(\tilde{u})\\ \widehat\sigma(u) < \widehat\sigma(d(u))}}{a(u)} \leq\\
\sum_{u=\pi^r(\tilde{u})}a(u)-\sum_{\substack{u=\pi^r(\tilde{u})\\ \widehat\sigma(u) < \widehat\sigma(d(u))}}{a(u)}-
\sum_{\substack{u=\pi^r(\tilde{u})\\ \widehat\sigma(u) > \widehat\sigma(d(u))}}{a(u)}=0,
\end{multline*}

Elements in $dg'(\mu)$ that are not of the type $\pi^r(u)$ are $(1,1), \dots, (r,1)$. 
We note that if a filling
is non-attacking then $\sigma(i,1)=i, i=1, \dots, r$. Hence these elements produce $\frac{r(r-1)}{2}$
inversions.
Lengths of their arms are $n~-~1, \dots, n-r$. Summing up we obtain $\frac{r(r-1)}{2}+r(n-r).$
This completes the proof of Lemma.
\end{proof}

Lemma \ref{inequalitypoweroft} tells us that $T(\widehat\sigma)$ is less than or equal to zero. 
At the $t\to \infty$ limit  only the fillings with  $T(\widehat\sigma)=0$ survive.
Any three elements considered in the previous Lemma give one negative summand to the power of $t$.
Therefore if $\widehat\sigma(u)\ge\widehat\sigma(d(u))$ then for any element $v$ in the arm of
$u$ only one of the inequalities $\widehat\sigma(v)<\widehat\sigma(u)$ and 
$\widehat\sigma(v)>\widehat\sigma(d(u))$
holds, $\widehat\sigma(v)$ is not between $\widehat\sigma(u)$ and $\widehat\sigma(d(u))$. 
And if $\widehat\sigma(u)<\widehat\sigma(d(u))$ then for any $v$ in the arm of
$u$ non of these equations hold, i. e. $\widehat\sigma(v)$ is between $\widehat\sigma(u)$ and $\widehat\sigma(d(u))$.
We thus obtain the following description of fillings such that the power of $t$ \eqref{degt} vanishes.

\begin{prop}\label{ruleoffilling}
Let $\widehat\sigma$ be a non-attacking filling of the diagram $\widehat{dg}(\mu)$ such that
$T(\widehat{\sigma})=0$. Let $u\in dg'(\mu)$, $\widehat\sigma(d(u))=k$ and
let $S=\lbrace \sigma(v)| v \in arm(u) \rbrace \cup \lbrace \sigma(u) \rbrace.$ 
Then 
\[
\sigma(u)=\begin{cases}
\min_{x \in S, x \geq k}x, \text{ if } \exists x \in S, x \geq k,\\
\min_{x \in S} x,\text{ otherwise}.
\end{cases}
\]
\end{prop}

Using the previous proposition we can obtain the following description of all
fillings of $\widehat{dg}(\mu)$, $\mu=\pi^r(\la)$, such that $T(\widehat\sigma)=0$.
We denote by $dg'_j(\la)$ the $j$-th row of the diagram.

\begin{lem}
Assume that $T(\widehat\sigma)=0$. Then the filling of
$\pi^r(dg'_j(\la))$ and the set $S=\{\sigma(u), u\in \pi^r(dg'_{j+1}(\la))\}$ 
uniquely determines the filling of $\pi^r(dg'_{j+1}(\la))$. 
We fill $\pi^r(dg'_{j+1}(\la))$ cell by cell from $\pi^r(n,j+1)$ to $\pi^r(1,j+1)$. 
Given a cell $v\in \pi^r(dg'_{j+1}(\la))$, the value $\sigma(v)$ is determined as follows:
let $S'$ be the set of elements of $S$ that are not used at the previous steps. 
Then $\sigma(v)$ is equal to:

(i) $\min\lbrace a \in S', a \geq \sigma(d(v))\rbrace$, if $\lbrace a \in S', a \geq \sigma(d(v))\rbrace\neq \emptyset$;

(ii) $\min\lbrace a \in S' \rbrace$, if $\lbrace a \in S', a \geq \sigma(d(v))\rbrace= \emptyset$.

\end{lem}

\begin{proof}
Immediate consequence of Proposition \ref{ruleoffilling}.
\end{proof}

In what follows we call the fillings $\sigma(\la)$ such that $T(\widehat\sigma)=0$ {\it appropriate}. 

\begin{prop}\label{fillingdegq}
Let $\la=(\lambda_1\leq \lambda_2 \leq \dots \leq \lambda_{n})$ be an antidominant weight.
Let $\widehat\sigma$ be a filling of the diagram $\widehat{dg}(\la)$ such that
the elements of the lowest row are $\sigma(i,0)=i$. Then there is only one appropriate 
filling with the fixed sets of elements in every row. For such a filling $\widehat\sigma$
one has
\[
Q(\sigma)=-\sum_{\widehat\sigma(u)<\widehat\sigma(d(u))} (l(u)+1).
\]
\end{prop}

In particular, we obtain the following well-known corollary:
\begin{cor}
Let $m_i=\la_{n+1-i}-\la_{n-i}$. Then 
\[
E_{\lambda}(x; 1, \infty)=\ch \bigotimes_{i=1}^{n-1} V_{\om_i}^{m_{i+1}-m_i},
\]
where $V_{\om_i}$ are fundamental representations.
\end{cor}

\begin{proof}
An appropriate filling $\widehat\sigma$ contains a set of different elements in each row of $dg'(\la)$. 
Thanks to Proposition \ref{fillingdegq} we have a bijection
between rows of our filling and elements of the standard basis of $V_{\om_k}$. The weight of the corresponding 
element of the basis equals to $\prod_{i=k+1}^n x^{\sigma(i,j)}$.
\end{proof}

\subsection{Recurrent formula}
Let $\la$ be an antidominant composition such that 
$\lambda_1= \dots = \lambda_{n-s}=0 \neq \lambda_{n-s+1}$
(i.e. there are $s$ cells in $dg'(\la)$).
Let $\ba=(a_{n-s+1}, \dots, a_n)$ be elements of the lowest row of an appropriate filling $\sigma$,  
i.e. the rule of Proposition \ref{ruleoffilling} is satisfied (we assume that 
$a_j$ lives in the $j$-th column). 
We denote the lowest row of $\sigma$ by ${\rm low}(\sigma)$ and define 
\[
k(\sigma)=(k_1, \dots, k_n),\ k_i=|\lbrace u|\sigma(u)=i \rbrace|.
\]
Recall that 
\[
E_\la(x,q^{-1},\infty)=\sum_{\sigma:\ appropriate} q^{Q(\sigma)}x^{k(\sigma)}.
\]

We introduce the notation which we extensively use below: 
\begin{equation}\label{definitionc}
c_{\ba}^{\lambda}(\bk)=
\sum_{\substack{\sigma:\ {\rm low}(\sigma)=(a_1, \dots, a_s),\\ k(\sigma)=(k_1, \dots, k_n)}} q^{Q(\sigma)}.
\end{equation}

Using Proposition \ref{ruleoffilling} we have:
\begin{multline*}
E_{\lambda}(x; q^{-1},\infty)=\\
\sum_{k_1, \dots, k_n \ge 0} c_{(n-s+1, \dots, n)}^{(0, \dots, 0, \lambda_{n-s+1}+1, \dots, \lambda_n+1)}
(k_1, \dots, k_{n-s}, k_{n-s}+1, \dots, k_n+1)x_1^{k_1} \dots  x_n^{k_n}.
\end{multline*}

Let $s'$ be the number of cells in the lowest but one row. For a string $\ba=(a_{n-s+1}, \dots, a_n)$ 
and a set $B$, $|B|=s'\le s$,
let $\ba\left(B\right)=(b_{n-s'+1},\dots,b_n)$
be an ordering of $B$ obtained using the rule of Lemma  \ref{ruleoffilling}.

\begin{prop}\label{mostgeneralrecurrence}
Let $\delta_1,\dots,\delta_n$ be the numbers defined by 
\[
\delta_i=\begin{cases} 1, \text{ if } i \in \lbrace a_{n-s+1}, \dots, a_n \rbrace,\\ 0 \text{ otherwise}. \end{cases} 
\]
Then 
\begin{multline*}
c_{\ba}^{(0, \dots, 0, \lambda_{n-s+1}+1, \dots, \lambda_n+1)}
(k_1+\delta_1, \dots, k_n+ \delta_n)=\\
\sum_{B:\ |B|=s'}c_{\ba(B)}^{\lambda}(k_1, \dots, k_n)q^{\sum_{j:\ b_j <a_j}\lambda_j}.
\end{multline*}
\end{prop}
\begin{proof}
This is an immediate consequence of Theorem \ref{fillingdegq}.
\end{proof}

\begin{example}
Consider the case $n=3$ and partitions of the type $\lambda=(0,m_2,m_1+m_2)$ (this is the general case for $\msl_3$).
Then we obtain the following equations:
\[
c_{(2,1)}^{(0,m_2+1,m_1+m_2+1)}(k_1+1, k_2+1, k_3)=c_{(2,1)}^{(0,m_2,m_1+m_2)}(k_1, k_2, k_3)+\]
\[+c_{(3,1)}^{(0,m_2,m_1+m_2)}(k_1, k_2, k_3)+c_{(3,2)}^{(0,m_2,m_1+m_2)}(k_1, k_2, k_3);
\]
\[
c_{(3,1)}^{(0,m_2+1,m_1+m_2+1)}(k_1+1, k_2, k_3+1)=c_{(2,1)}^{(0,m_2,m_1+m_2)}(k_1, k_2, k_3)q^{m_2}+\]
\[+c_{(3,1)}^{(0,m_2,m_1+m_2)}(k_1, k_2, k_3)+c_{(3,2)}^{(0,m_2,m_1+m_2)}(k_1, k_2, k_3).
\]

Using these equations we obtain:
\begin{multline}\label{c31forsl3}
c_{(3,1)}^{(0,m_2+1,m_1+m_2+1)}(k_1+1, k_2, k_3+1)=\\
c_{(2,1)}^{(0,m_2+1,m_1+m_2+1)}(k_1+1, k_2+1, k_3)-
(1-q^{m_2})c_{(2,1)}^{(0,m_2,m_1+m_2)}(k_1, k_2, k_3).
\end{multline}
\end{example}

\begin{prop}\label{transpositioncolumns}
Let $\delta_j$ be as in Proposition \ref{mostgeneralrecurrence} and assume that $\lambda_{n-i}=\lambda_{n-i+1}$ for some $i$. 
Then:
\begin{multline*}
c_{(a_{n-s+1}, \dots,a_{n-s+i}, a_{n-s+i+1}, \dots, a_n)}^{(0, \dots, 0, \lambda_{n-s+1}+1, \dots, \lambda_n+1)}
(k_1+\delta_1, \dots, k_n+ \delta_n)=\\
c_{(a_{n-s+1}, \dots,a_{n-s+i+1}, a_{n-s+i}, \dots, a_n)}^{(0, \dots, 0, \lambda_{n-s+1}+1, \dots, \lambda_n+1)}
(k_1+\delta_1, \dots, k_n+ \delta_n).
\end{multline*}
\end{prop}
\begin{proof}
We prove proposition  by induction on the number of cells in $\lambda$ using Proposition \ref{mostgeneralrecurrence}.

Assume that the proposition is proved for all partitions $\lambda'$ such that $|\lambda'|<|\lambda|$.
For $\ba=(a_{n-s+1}, \dots,a_{n-i}, a_{n-i+1}, \dots, a_n)$ we define 
\[
t_{n-i}\ba=(a_{n-s+1}, \dots,a_{n-i+1}, a_{n-i}, \dots, a_n).
\]
For a subset $B\subset \lbrace 1, \dots, n \rbrace$, $|B|=s' \leq s$ we define 
\[
\ba(B)=(b_{n-s'+1},\dots, b_n),\quad (t_{n-i}\ba)(B)=(\tilde b_{n-s'+1},\dots, \tilde b_n)
\]
Using Proposition \ref{mostgeneralrecurrence} we have:
\begin{multline*}
c_{\ba}^{(0, \dots, 0, \lambda_{n-s+1}+1, \dots, \lambda_n+1)}
(k_1+\delta_1, \dots, k_n+ \delta_n)=\\
\sum_{B:\ |B|=s'}c_{\ba(B)}^{\lambda}(k_1, \dots, k_n)q^{\sum_{j:\ b_j <a_j}\lambda_j};\\
\shoveleft c_{t_{n-i}\ba}^{(0, \dots, 0, \lambda_{n-s+1}+1, \dots, \lambda_n+1)}
(k_1+\delta_1, \dots, k_n+ \delta_n)=\\
\sum_{B:\ |B|=s'}c_{(t_{n-i}\ba)(B)}^{\lambda}(k_1, \dots, k_n)q^{\sum_{j:\ \tilde b_j < (t_{n-i}\ba)_j}\lambda_j}.
\end{multline*}

Our goal is to prove that 
\[
c_{(t_{n-i}\ba)(B)}^{\lambda}(k_1, \dots, k_n)q^{\sum_{j:\ b_j <a_j}\lambda_j}=
c_{\ba(B)}^{\lambda}(k_1, \dots, k_n)q^{\sum_{j:\ \tilde b_j <(t_{n-i}\ba)_j}\lambda_j}
\]
for any $B$. More precisely, we will prove that the powers of $q$ are equal and that 
\begin{equation}\label{cc}
c_{(t_{n-i}\ba)(B)}^{\lambda}(k_1, \dots, k_n)=c_{\ba(B)}^{\lambda}(k_1, \dots, k_n).
\end{equation}

We claim that there are two possibilities: either $(t_{n-i}\ba)(B)=\ba(B)$ or 
$(t_{n-i}\ba)(B)=(t_{n-i}(\ba(B))$ (we note that this claim implies the equality \eqref{cc} by induction
on $|\la|$). First, we note that  
$b_j=\tilde b_j$ for $j\ge n-i+2$. We denote by $B'$ the set $B\setminus \{b_{n-i+2},\dots,b_n\}$.
To prove the desired claim, we consider two cases. First, assume that either 
\begin{equation}\label{either}
b_j<\min(a_{n-i},a_{n-i+1}) \text{ or } b_j\ge \max(a_{n-i},a_{n-i+1}) \text{ for all } j
\end{equation}
or 
\begin{equation}\label{or}
\min(a_{n-i},a_{n-i+1})\le b_j <\max(a_{n-i},a_{n-i+1}) \text{ for all } j.
\end{equation}
Then $(t_{n-i}\ba)(B)=\ba(B)$ thanks to Proposition \ref{ruleoffilling}. 
Moreover, 
$\sum_{j:\ b_j <a_j}\lambda_j = \sum_{j:\ \tilde b_j <(t_{n-i}\ba)_j}\lambda_j$.
In fact, we know that $\la_{n-i}=\la_{n-i+1}$ and one of the conditions \eqref{either}, \eqref{or}
is satisfied. Hence the sums above are equal. 

Now assume that both conditions \eqref{either}, \eqref{or} are not satisfied. Then
$(t_{n-i}\ba)(B)=(t_{n-i}(\ba(B))$ because of Proposition \ref{ruleoffilling}. In particular,
$\sum_{j:\ b_j <a_j}\lambda_j = \sum_{j:\ \tilde b_j <(t_{n-i}\ba)_j}\lambda_j$ because the summands
of the sums coincide.
\end{proof}

\begin{prop}\label{specialisationinzero}
For antidominant $\la$ one has
\[
\ch_{KM} W_\lambda(x_1, \dots, x_n, q)=\sum c^{\lambda}_{(s, s-1, \dots, 1)}(k_1+1, \dots, k_s+1, k_{s+1}, \dots, k_n)
x_1^{k_1}\dots x_n^{k_n}.\]
\end{prop}
\begin{proof}
We know (see \cite{S}) that $\ch_{KM} W_\lambda(x_1, \dots, x_n, q)=E_{\lambda}(x; q, 0)$.
Now we compute $E_{\lambda}(x;q,0)$ using the Haglund-Haiman-Loehr formula.
One has
\[
E_{\lambda}(x; q, 0)=\sum_{\sigma ~ \rm non-attacking \it}x^{\sigma}q^{maj(\sigma)}0^{coinv(\widehat\sigma)}.
\]
We note that $coinv(\sigma) \geq 0$. Indeed,
\[
coinv(\widehat\sigma)=\sum_{u \in dg'}a(u)-|Inv'(\widehat\sigma)|+\sum_{u \in Des(\widehat\sigma)}a(u).
\]
For any two boxes $u, u' \in leg(u)$ we have that if $\sigma(u)>\sigma(u')$ and $\sigma(u')>\sigma(d(u))$
then $\sigma(u)>\sigma(d(u))$, so similarly to the proof of Proposition \ref{ruleoffilling} we have that
$coinv(\widehat\sigma) \geq 0$ and $coinv(\widehat\sigma) = 0$ if and only if $\sigma$ if obtained by
following inverse rule of filling:

Assume that we have filled the $i$-th row. Let $S$ be the set of elements of the $(i+1)$-st row.
We fill the $(i+1)$-st row of the diagram from right to left. If $S'$ is the
set of elements of $S$ that are not used before, then into the cell $v$ we put:

(i) $\max\lbrace a \in S', a \leq \widehat\sigma(d(v))\rbrace$, if $\lbrace a \in S', a \leq \widehat\sigma(d(v))\rbrace\neq \emptyset$;

(ii) $\max\lbrace a \in S' \rbrace$, if $\lbrace a \in S', a \geq \widehat\sigma(d(v))\rbrace= \emptyset$.

We conclude that 
\[
E_{\lambda}(x; q, 0)=\sum c^{\lambda}_{(s, s-1, \dots, 1)}(k_1+1, \dots, k_s+1, k_{s+1}, \dots, k_n)
x_1^{k_n}\dots x_n^{k_1}.
\]
Hence
\[
\sum c^{\lambda}_{s, s-1, \dots, 1}(k_1+1, \dots, k_s+1, k_{s+1}, \dots, k_n)
x_1^{k_n}\dots x_n^{k_1}=\ch_{KM} W_\lambda(x_1, \dots, x_n, q).
\]
But $\ch_{KM} W_\lambda(x_1, \dots, x_n, q)$ is a symmetric function in $x_i$'s. Therefore we obtain:
\[
\sum c^{\lambda}_{(s, s-1, \dots, 1)}(k_1+1, \dots, k_s+1, k_{s+1}, \dots, k_n)
x_1^{k_1}\dots x_n^{k_n}=\ch_{KM} W_\lambda(x_1, \dots, x_n, q).
\]
\end{proof}

\subsection{Scrolled case}
Recall the mappings
\[\pi(\lambda_1, \dots, \lambda_n)=(\lambda_n+1, \lambda_1, \dots, \lambda_{n-1});\]
\[\Psi f(x_1, \dots, x_n)=x_1 f(x_2, \dots, x_n, q^{-1}x_1).\]
The Knop-Sahi recurrence states that 
\begin{equation}\label{psionMacdonald}
E_{\pi(\mu)}(x; q, t)=q^{\mu_n} \Psi E_{\mu}(x; q, t).
\end{equation}

\begin{prop}\label{interpretationscrollledc}
Consider a partition $\lambda=(\lambda_1, \dots, \lambda_n)$, $\lambda_1 \leq \dots \leq \lambda_n$,
$\lambda_1= \dots = \lambda_{n-s} =0 \neq \lambda_{n-s+1}$.
Then
\begin{multline*}
E_{\pi^r \lambda}(x; q^{-1}, \infty)=\sum_{k_1,\dots,k_n\ge 0} x_1^{k_1+1}\dots x_r^{k_r+1}x_{r+1}^{k_{r+1}}\dots x_n^{k_n} \times\\
c_{(n-s+r+1, \dots, n, 1, \dots, r)}^{(0, \dots, 0, \lambda_{n-s+1}+1, \dots, \lambda_n+1)}
(k_1+1, \dots, k_r+1, k_{r+1},  \dots, k_{n-s+r}, k_{n-s+r+1}+1, \dots,  k_n+ 1).
\end{multline*}
\end{prop}
\begin{proof}
It is an immediate consequence of Proposition \ref{ruleoffilling} for partition $\pi^r(\lambda)$. 
\end{proof}

\begin{lem}\label{equalityofc}
\begin{multline*}
q^{\lambda_{n-r+1}+ \dots + \lambda_n-k_{n-r+1}- \dots -k_n} \times\\
c_{(n-s+r+1, \dots, n, 1, \dots, r)}^{0, \dots, 0, \lambda_{n-s+1}+1, \dots, \lambda_n+1}
(k_{n-r+1}+1, \dots, k_n+1, k_{1}+1,  \dots, k_{r}+1, k_{r+1}, \dots,  k_{n-s+r}+ 1)=\\
c_{(n-s+1, \dots, n)}^{0, \dots, 0, \lambda_{n-s+1}+1, \dots, \lambda_n+1}
(k_1+1, \dots, k_r+1, k_{r+1},  \dots, k_{n-s+r}, k_{n-s+r+1}+1, \dots,  k_n+ 1)
\end{multline*}
\end{lem}

\begin{proof}
Immediate consequence of  equation \eqref{psionMacdonald} and  Proposition \ref{interpretationscrollledc}.
\end{proof}

\begin{example}
Consider the case $n=3$ and  partition $\lambda=(0, m_2, m_1+m_2)$.
Then by Proposition \ref{interpretationscrollledc} we have:
\[
E_{(m_1+m_2+1, 0, m_2)}(x_1, x_2, x_3; q^{-1}, \infty)=
x_1\sum_{k_1,k_2,k_3\ge 0} c^\la_{(3,1)}(k_1, k_2, k_3)x_1^{k_1}x_2^{k_2}x_3^{k_3}.
\]
Using equation \eqref{psionMacdonald} we obtain
\[
c^\la_{(2,3)}(k_1, k_2+1, k_3+1) = q^{m_1+m_2-k_3}c^\la_{(3,1)}(k_3+1, k_1, k_2+1).
\]
\end{example}

\section{The Cherednik-Orr conjecture}
Let $\fg=\msl_n$. If $\la=\sum_{i=1}^{n-1} m_i\om_i$, then the corresponding diagram is equal to
$\la=(\la_1,\dots,\la_n)$, where $\la_i=m_i+\dots+m_{n-1}$ ($\la_n=0$). The diagram
of the antidominant weight $\la^*$ is reversed, i.e. is given by $(\la_n,\dots,\la_1)$.
For example, the diagram corresponding to $(m\om_r)^*$ is equal to $(\underbrace{0,\dots,0}_{n-r},m,\dots,m)$.
 
\subsection{Rectangular diagrams}
In this subsection we use Proposition \ref{rectprop}. 
Assume that the diagram $dg'(\la)$ is rectangular of length $r$ and height $m$. By Proposition \ref{ruleoffilling}
we know that an appropriate filling is completely determined by $m$ subsets of $r$ elements from $1, \dots, n$.
We know that an order of elements in $i$-th row determines the order of elements in $(i+1)$-st row.
Because of Proposition \ref{transpositioncolumns} the order of elements $a_{n-r+1},\dots,a_n$
in 
\begin{equation}\label{rect}
c_{(a_{n-r+1},\dots,a_n)}^{(0, \dots, 0, m+1, \dots, m+1)}
(k_1+\delta_1, \dots, k_n+ \delta_n)
\end{equation}
($\delta_i$ as in Proposition \ref{mostgeneralrecurrence})
is not important, we only care about the set $\{a_{n-r+1},\dots,a_n\}$.
We use Proposition \ref{mostgeneralrecurrence} to write recurrent equations for
\eqref{rect}. 

\begin{example} Let $r=2$, $i<j$. Then 
\begin{multline*}
c^\la_{\{i,j\}}(k_1, \dots,k_{i-1}, k_i+1,k_{i+1}, \dots k_{j-1}, k_j+1, k_{j+1}, \dots, k_n)=\\
q^{2m}\sum_{p,q<i}c^\la_{\{p,q\}}(k_1,\dots, k_n)
+q^m\sum_{p<i, q \geq i}c^\la_{\{p,q\}}(k_1,\dots, k_n)+\\
q^m\sum_{p \geq i, q <j}c^\la_{\{p,q\}}(k_1,\dots, k_n)+\sum_{p\geq i, q \geq j}c^\la_{\{p,q\}}(k_1,\dots, k_n)
\end{multline*}
\end{example}

Note that the most important case for us is $c^\la_{\{n-r+1, \dots, n\}}$.
Using Proposition \ref{interpretationscrollledc} we have:
\begin{equation}
E_{\pi^r(\lambda)}(x,q^{-1},infty)=x_1 \dots x_r \sum_{k_i\ge 0} c^\la_{\lbrace 1 \dots r \rbrace}
(k_1+1,\dots, k_r+1, k_{r+1}, \dots, k_n)x_1^{k_1}\dots x_n^{k_n}.
\end{equation}
We note that $\pi^r(\lambda)=(m+1, \dots, m+1, 0, \dots, 0)$.

\begin{thm}\label{rectangular}
Let $\lambda=(m\om_r)^*$. Then we have:
\[
E_{\lambda}(x;q^{-1}, \infty)=\ch_{PBW}{W_{\lambda}^{gr}}(x; q, q).
\]
\end{thm}
\begin{proof}
Using Proposition \ref{specialisationinzero} we have:
\[\ch_{KM} W_\lambda(x_1, \dots, x_n, q)=\sum c^\la_{\{1, 2, \dots, r\}}
(k_1+1, \dots, k_r+1, k_{r+1}, \dots, k_n)
x_1^{k_1}\dots x_n^{k_n}.\]

Now using Lemma \ref{equalityofc} 
we obtain:
\begin{multline*}
q^{rm-k_{n-r+1}-\dots -k_n}c^\la_{\lbrace 1, 2, \dots, r\rbrace}
(k_1+1, \dots, k_r+1, k_{r+1}, \dots, k_n)=\\
c^\la_{\lbrace n-r+1, \dots, n\rbrace}
(k_1, \dots, k_{n-r}, k_{n-r+1}+1, \dots, k_n+1).
\end{multline*}
So we have:
\begin{multline*}
E_{\lambda}(x;q^{-1}, \infty)=\\
\sum c^\la_{\lbrace n-r+1, \dots, n\rbrace}
(k_1, \dots, k_{n-r}, k_{n-r+1}+1, \dots, k_n+1)x_1^{k_1} \dots x_n^{k_n}=\\
=\sum q^{rm-k_{n-r+1}-\dots -k_n}c^\la_{\lbrace 1, 2, \dots, r\rbrace}
(k_1+1, \dots, k_r+1, k_{r+1}, \dots, k_n)x_1^{k_1} \dots x_n^{k_n}=\\
\sum c^\la_{\lbrace 1, 2, \dots, r\rbrace}
(k_1+1, \dots, k_r+1, k_{r+1}, \dots, k_n)(qx_1)^{k_1}\dots (qx_{n-r})^{k_{n-r}}x_{n-r+1}^{k_{n-r+1}} \dots x_n^{k_n}=\\
\ch_{KM} W_\lambda(qx_1, \dots, qx_{n-r}, x_{n-r+1}, \dots, x_n, q).
\end{multline*}
Then using Proposition \ref{rectprop} we complete the proof of this Theorem.
\end{proof}

\subsection{$m_1\om_1+m_2\om_{n-1}$-case}
Let us fix the highest weight $m_1\om_1+m_2\om_{n-1}$ (for $n=3$ this is the general case).
Then the corresponding diagram is of the form $(0,m_2,\dots,m_2,m_1+m_2)$.
Let $A=(a_1, \dots, a_{n-1})$ be a string of $n-1$ different elements from
the set $\{1, \dots, n\}$, $\lbrace 1, \dots, n \rbrace \backslash\lbrace a_1, \dots, a_{n-1} \rbrace =\lbrace \tilde{a} \rbrace$.
Let $t_i$ be the transposition of $i$-th and $(i+1)$-st elements.
Then Proposition \ref{transpositioncolumns} tells us that
\begin{multline*}
c_A^\lambda(k_1+1, \dots, k_{\tilde{a}-1}+1,k_{\tilde{a}}, k_{\tilde{a}+1}+1, \dots, k_{n}+1)=\\
c_{t_iA}^\lambda(k_1+1, \dots, k_{\tilde{a}-1}+1,k_{\tilde{a}}, k_{\tilde{a}+1}+1, \dots, k_{n}+1),
\end{multline*}
for $1 \leq i \leq n-3$. So the only essential parameters are $\tilde{a}$ and $a_{n-1}$
(we can not interchange $a_{n-1}$ with $a_{n-2}$, because $\la_n>\la_{n-1}$). 
We denote these polynomials by
$c_{\tilde{a}|a_{n-1}}^{\lambda}(k_1, \dots, k_n)$.

\begin{prop}\label{characterfor1n}
Assume that $\lambda^*=m_1 \om_1+ m_2 \om_{n-1}$.
Then: 
\begin{multline*}
\ch_{KM} W_{\lambda}(x,q)=\\
\sum_{k_i \geq 0\ }
\sum_{\substack{p_i+\sum_{\alpha \neq i}l_{\alpha}=k_i,\\ 
l_1 + \dots +l_n=m_2,\\ p_1 + \dots + p_n=m_1}}
q^{l_1p_1}\binom{m_2}{l_1, \dots, l_{n}}_q\binom{m_1}{p_1, \dots, p_{n}}_q x_1^{k_1} \dots x_n^{k_n}.
\end{multline*}
\begin{multline*}
\ch_{PBW} W_{\lambda}^{gr}(x,q,q)=\\
\sum_{k_i \geq 0\ }
\sum_{\substack{p_i+\sum_{\alpha \neq i}l_{\alpha}=k_i,\\ l_1+\dots +l_n=m_2,\\ p_1 + \dots + p_n=m_1}}
q^{l_1p_1+l_2+ \dots + l_n +p_1 +\dots +p_{n-1}}
\binom{m_2}{l_1, \dots, l_{n}}_q\binom{m_1}{p_1, \dots, p_{n}}_q x_1^{k_1} \dots x_n^{k_n}.
\end{multline*}
\end{prop}
\begin{proof}
We construct a bijection $\eta$ between elements of the PBW-basis for $W_{\lambda}$ and pairs of strings:
one of them of the length $m_1$ and other of the length $m_2$, both filled by elements of the set
$\lbrace 1, \dots, n \rbrace$. We denote the pairs of filled strings by
$\tau=a_1, \dots, a_{m_2}|b_1, \dots, b_{m_1}$.
Put
\[\eta(\tau)=
\prod_{a_i=n} e_{1,n-1}\otimes t^{|\lbrace j<i| a_j=1 \rbrace|}
\prod_{i|a_i \neq 1,n}e_{1,a_i-1}\otimes t^{|\lbrace j<i| a_j<a_i ~ \rm or\ \it a_j=n \rbrace|}\]
\[\prod_{b_i=1} e_{1,n-1}\otimes t^{|\lbrace j<i| a_j=n \rbrace|+|\lbrace  a_j=1 \rbrace|}
\prod_{i|b_i \neq 1,n}e_{b_i-1,n-1}\otimes t^{|\lbrace j<i| b_j<b_i ~ \rm or\ \it b_j=n \rbrace|}.\]

Note that for any $\tau$ $\eta(\tau)$ satisfies equalities \eqref{1}, \eqref{n}, \eqref{1,n} and
comparing the numbers of elements in both sets we obtain that $\eta$ is indeed a bijection. 
Put 
\begin{multline*}
d(\tau)=\big|\lbrace (i<j)| a_i< a_j<n\rm ~or \it ~ a_i=n,a_j <n \rbrace\big|+\\
\big|\lbrace (i<j)| b_i< b_j<n\rm ~or \it ~ b_i=n,b_j <n \rbrace\big|+
\big| \lbrace i| a_i=1, b_i=1 \rbrace \big|.
\end{multline*} 
Then by definition of $\eta$ we have $d(\eta(\tau))=d(\tau)$, where $d$ in the right hand side is the
Kac-Moody energy operator.

Note that $$\deg_{PBW}(\eta(\tau))=\big| \lbrace i| a_i \neq 1 \rbrace \big|+
\big| \lbrace i| b_i \neq n \rbrace \big|$$ and
its weight is $(k_1, \dots, k_n)$, where $k_i = p_i + \sum_{\alpha \neq i}l_{\alpha}$, and
$p_i=\big| \lbrace j| b_j=i \rbrace \big|$, $l_i=\big| \lbrace j| a_j=i \rbrace \big|$.

Fix $l_i$ and $p_i$. Then sum of $q^{d(\tau)}$ for such elements is
$q^{l_1p_1}\binom{m_2}{l_1, \dots, l_{n}}_q\binom{m_1}{p_1, \dots, p_{n}}_q$.
Indeed, the last summand in the definition of $d(\tau)$ is $l_1p_1$ and by definition
of the q-binomial coefficients: 
\[\binom{m_2}{l_1, \dots, l_{n}}_q=
\sum q^{\big| \lbrace i<j|a_i< a_j<n\rm ~or \it ~ a_i=n,a_j <n\ \rbrace \big|}.\]
\[\binom{m_1}{p_1, \dots, p_{n}}_q=
\sum q^{\big| \lbrace i<j|b_i< b_j<n\rm ~or \it ~ b_i=n,b_j <n\ \rbrace \big|}.\]

Similarly we obtain second equation of the Proposition. This completes the proof of the Lemma.
\end{proof}

\begin{thm}
The Cherednik-Orr conjecture is true for $\lambda^*=m_1 \om_1+m_2\om_{n-1}$
\end{thm}
\begin{proof}
Using Proposition \ref{specialisationinzero} and Theorem \ref{main} we have:
\[c^{(0,m_2+1, \dots, m_2 +1, m_2+m_1+1)}_{n|1}(k_1+1, \dots, k_{n-1}+1, k_n)=\]
\[=\sum_{\substack{p_i+\sum_{\alpha \neq i}l_{\alpha}=k_i\\
l_1 +  \dots +l_n=m_2,\\ p_1 + \dots + p_n=m_1}}
q^{l_1p_1}\binom{m_2}{l_1, \dots, l_{n}}_q\binom{m_1}{p_1, \dots, p_{n}}_q.\]

Using Proposition \ref{mostgeneralrecurrence} we have for $j\ge 2$:

\[c^{(0,m_2+1, \dots, m_2 +1, m_2+m_1+1)}_{j|1}(k_1+1, \dots, k_{j-1}+1, k_j, k_{j+1}+1, \dots,  k_n+1)=\]
\[\sum_{i=2}^{j} c^{\lambda}_{i|1}(k_1, \dots k_n)+\sum_{i=j+1}^{n} c^{\lambda}_{i|1}(k_1, \dots k_n)q^{m_2}+
c^{\lambda}_{1|2}(k_1, \dots k_n).\]

Subtracting the equations for consequent $j$'s we obtain:

\begin{multline*}
c^{(0,m_2+1, \dots, m_2 +1, m_2+m_1+1)}_{j|1}(k_1+1, \dots, k_{j-1}+1, k_j, k_{j+1}+1, \dots,  k_n+1)=\\
c^{(0,m_2+1, \dots, m_2 +1, m_2+m_1+1)}_{j+1|1}(k_1+1, \dots, k_{j}+1, k_{j+1}, k_{j+2}+1, \dots,  k_n+1)-\\
(1-q^{m_2})c^{\lambda}_{j+1|1}(k_1, \dots, k_n).
\end{multline*}

We claim that
\begin{multline*}
c^{(0,m_2+1, \dots, m_2 +1, m_2+m_1+1)}_{j|1}(k_1+1, \dots, k_{j-1}+1, k_j, k_{j+1}+1, \dots,  k_n+1)=\\
\sum_{\substack{p_i+l_{i+1} + \dots + l_n +l_1 + \dots + l_{i-1}=k_i,\\ 
l_1+\dots +l_n=m_2,\\ p_1 + \dots + p_n=m_1}}
q^{l_1p_1+l_{j+1} + \dots +l_n}\binom{m_2}{l_1, \dots, l_{n}}_q\binom{m_1}{p_1, \dots, p_{n}}_q.
\end{multline*}

Indeed, we know that it is true for $j=n$ and
\begin{multline*}
\sum_{\substack{p_i+\sum_{\alpha \neq i}l_{\alpha}=k_i,\\ l_1 + \dots +l_n=m_2 \\ p_1 + \dots + p_n=m_1}}
q^{l_1p_1+l_{j+1} + \dots +l_n}\binom{m_2}{l_1, \dots, l_{n}}_q\binom{m_1}{p_1, \dots, p_{n}}_q -\\
(1-q^{m_2})\sum_{\substack{
p_i+\sum_{\alpha \neq i}l_{\alpha}=k_i-1, i \neq j,\\
p_j+\sum_{\alpha \neq j}l_{\alpha}=k_j,\\l_1+\dots +l_n=m_2-1,\\ p_1 + \dots + p_n=m_1}}
q^{l_1p_1+l_{j+1} + \dots +l_n}\binom{m_2}{l_1, \dots, l_{n}}_q\binom{m_1}{p_1, \dots, p_{n}}_q=\\
\sum_{\substack{p_i+\sum_{\alpha \neq i}l_{\alpha}=k_i,\\ l_1 + \dots +l_n=m_2,\\ p_1 + \dots + p_n=m_1}}
\Bigg( q^{l_1p_1+l_{j+1} + \dots +l_n}\binom{m_2}{l_1, \dots, l_{n}}_q\binom{m_1}{p_1, \dots, p_{n}}_q- \\
(1-q^{m_2})q^{l_1p_1+l_{j+1} + \dots +l_n}\binom{m_2-1}{l_1, \dots, l_{j-1},
 l_j-1, l_{j+1}, \dots, l_{n}}_q\binom{m_1}{p_1, \dots, p_{n}}_q\Bigg)=\\
\sum_{\substack{p_i+\sum_{\alpha \neq i}l_{\alpha}=k_i,\\ l_1+\dots +l_n=m_2,\\ p_1 + \dots + p_n=m_1}}
q^{l_1p_1+l_{j+1} + \dots +l_n+l_j}\binom{m_2}{l_1, \dots, l_{n}}_q\binom{m_1}{p_1, \dots, p_{n}}_q
\end{multline*}
 In particular we obtain:
\begin{multline*}
c^{(0,m_2+1, \dots, m_2 +1, m_2+m_1+1)}_{2|1}=\\
\sum_{\substack{p_i+\sum_{\alpha \neq i}l_{\alpha}=k_i,\\ l_1 + \dots +l_n=m_2,\\ p_1 + \dots + p_n=m_1}}
q^{l_1p_1+l_{3} + \dots +l_n}\binom{m_2}{l_1, \dots, l_{n}}_q\binom{m_1}{p_1, \dots, p_{n}}_q. 
\end{multline*}

But using Lemma \ref{equalityofc} we have:
\begin{multline*}
c^{(0,m_2+1, \dots, m_2 +1, m_2+m_1+1)}_{1|n}
(k_1, k_2+1,\dots,  k_n+ 1)=\\
q^{m_2+m_1 -k_n} c^{(0,m_2+1, \dots, m_2 +1, m_2+m_1+1)}_{2|1}
(k_{n}+1, k_1, k_2+1, \dots, k_{n-1}+1)=\\
\sum_{\substack{p_i+\sum_{\alpha \neq i}l_{\alpha}=k_{i-1(mod n)},\\ l_1+\dots +l_n=m_2,\\ p_1 + \dots + p_n=m_1}}
q^{l_1p_1+l_{3} + \dots +l_n+p_2+\dots + p_n +l_1}\binom{m_2}{l_1, \dots, l_{n}}_q\binom{m_1}{p_1, \dots, p_{n}}_q.\
\end{multline*}
We thus obtain exactly the coefficient of $x_1^{k_1} \dots x_n^{k_n}$
in the formula for PBW-character of Proposition \ref{characterfor1n}.
\end{proof}

\section*{Acknowledgments}
The work of Evgeny Feigin was partially supported
by the Dynasty Foundation, by the Simons foundation, by the AG Laboratory HSE, 
RF government grant ag. 11.G34.31.0023, by the RFBR grants
12-01-00070, 12-01-00944, 12-01-33101, 13-01-12401 and by the Russian Ministry of Education and Science under the
grant 2012-1.1-12-000-1011-016.

The work of  Ievgen Makedonskyi was partially supported
by the AG Laboratory HSE, RF government grant ag. 11.G34.31.0023.

\end{document}